\documentclass[12pt,a4paper]{article}

\usepackage[dvips]{graphicx}
\usepackage{amsmath}
\usepackage{amssymb,latexsym}
\usepackage{mathrsfs}
\usepackage{psfrag}
\usepackage[all]{xy}
\usepackage{color}
\usepackage{verbatim}
\usepackage{graphicx,geometry,color,epsfig,epstopdf}
\usepackage{caption}
\usepackage{subcaption}
\usepackage{setspace}
\usepackage{soul}

\newtheorem{prop}{Proposition}[section]
\newtheorem{theo}[prop]{Theorem}
\newtheorem{cor}[prop]{Corollary}
\newtheorem{ex}[prop]{Example}
\newtheorem{exs}[prop]{Examples}
\newtheorem{defn}[prop]{Definition}

\newtheorem{rems}[prop]{Remarks}

\newcommand{\C}{\mbox{\boldmath $\mathbb{C}$}}
\newcommand{\R}{\mbox{\boldmath $\mathbb{R}$}}
\newcommand{\K}{\mbox{\boldmath $\mathbb{K}$}}
\newcommand{\N}{\mbox{\boldmath $\mathbb{N}$}}
\newcommand{\BP}{\mbox{\boldmath $\mathbb{P}$}}

\newcommand{\ord}{\textrm{ord}}
\newcommand{\mult}{\textrm{mult}}

\newenvironment{proof}
{\begin{trivlist} \item[\hskip \labelsep {\bf Proof}\hspace*{3 mm}]}
	{\hfill$\Box$\end{trivlist}}
\newenvironment{acknow}
{\begin{trivlist} \item[\hskip \labelsep {\bf Acknowledgments.}]}
	{\end{trivlist}}

\begin{document}

\title{On vertices and inflections of singular plane  curves}
\author{J. W. Bruce, M. A. C. Fernandes and F. Tari}

\maketitle

\begin{abstract}
Given the germ of a smooth plane curve $(\{f(x,y)=0\},0)\subset (\K^2,0), \K=\R, \C$, with an isolated singularity, we define two invariants $I_f$ and $V_f\in \N\cup\{\infty\}$,  which count the number of inflections and vertices (suitably interpreted in the complex case) concentrated at the singular point. The first is an affine invariant, while the second is invariant under similarities of $\R^2$, and their analogue for $\C^2$. When the curve has no smooth components, these invariants are always finite and bounded. We illustrate our results by computing the range of possible values for these invariants for Arnold's ${\cal K}$-simple singularities. We also establish a relationship between these invariants, the Milnor number of $f$ and the contact of the curve germ with its \lq osculating circle\rq. 
\end{abstract}

\renewcommand{\thefootnote}{\fnsymbol{footnote}}

\section{Introduction}\label{sec:intro}
Let $C$ be germ of a plane curve in $\R^2$. The dual and evolute of $C$ are associated curves of significant interest. The dual has singularities corresponding to inflections and bi-tangencies, while the evolute has singularities related to vertices and bi-osculating circles.  Evolutes can also be interpreted as the caustics of the wavefront generated by $C$. Moreover, when $C$ is a projective plane curve, the dual of the dual is $C$.

When the curve is singular, we can still define both the dual and the evolute, and it is natural to ask what happens at the singular points. In particular, if we replace the singular curve with a nearby generic one, how many inflections and vertices emerge from the singularity? The answer naturally depends on how we define a generic deformation. If the curve is given parametrically, we deform the parametrisation to obtain an immersed curve with normal crossings. If it is given implicitly by an equation, we deform the equation to obtain a smooth curve. The question, in each case, is how many inflections and vertices emerge under these respective deformations. 
These questions have been addressed in \cite{DiasTari} for parametrised curves, and in \cite{WallDuality} for inflections, both in the case of parametrised curves and curves defined by equations. In this paper we give a self-contained account of all cases.

While primarily motivated by the real case, it is most efficient to start with germs of (possibly singular) holomorphic curves, defined either by a parametrisation $\gamma:(\C,0)\to (\C^2,0)$, or as a zero set of a germ $f:(\C^2,0)\to (\C,0)$, which we denote respectively by $(C_\gamma,0)$ and $(C_f,0)$. 

Inflections are points where some line has order of contact at least $3$; vertices where the order of contact with some circle or line is at least $4$. We define invariants $I_\gamma, I_f$ and $V_\gamma, V_f$ in $\N\cup\{\infty\}$, which count the number of inflections and vertices concentrated at the singular point. 

We show how to compute these numbers in terms of the corresponding invariants for the irreducible branches of $f$. Given a parametrisation $\gamma$ of an irreducible germ defined by $f=0$, we establish relations between $I_\gamma, I_f, \mu(f)$, $V_\gamma, V_f, \mu(f)$, where $\mu(f)$ is the Milnor number  of $f$. We also relate $I_{\gamma}, V_{\gamma}$ and the contact between $\gamma$ and its osculating circle. 

When $C_f$ has no smooth components, we show that $I_f$ and $V_f$ are bounded, and we provide detailed information for the range of values they can attain. As an illustration, we determine the possible values of $I_f$ and $V_f$ when $f$ has one of Arnold's simple singularities.

This paper is organized as follows:
\begin{itemize}
\item[--] In \S \ref{sec:prel}, we define the notion of vertices of regular complex analytic curves, discuss multiplicity and intersection numbers, and establish some results on the contact between germs of singular curves and smooth curves, which we use throughout the paper.

\item[--]  In \S \ref{sec:vert}, we define $I_\gamma$, $I_f$, $V_\gamma$ and $V_f$, and describe some of their properties. In particular, we relate $I_\gamma$ and $V_\gamma$ to the order of contact of the curve with its osculating line and osculating circle, respectively.

\item[--]  In \S\ref {sec:irreducible}, we relate $I_\gamma$ to $I_f$, and $V_\gamma$ to $V_f$, when the curve is irreducible, that is, when $(C_\gamma,0)=(C_f,0)$. We also determine the range of values that $I_f$ and $V_f$ can attain.

\item[--]  In \S\ref{sec:exam},  we carry out explicit calculations, including the ranges of values of  $I_f$ and $V_f$ for curves with simple singularities.

\item[--]  Finally, we return to the real case in \S \ref{sec:vertR}, deducing our results from the previous sections  using finite determinacy results. 
\end{itemize}

Throughout, all our defining functions  $f:(\C^2,0)\to (\C,0)$ are ${\cal K}$-finite, that is $f$ has no repeated factors.


\section{Preliminaries} \label{sec:prel}

We begin with some basic notions from singularity theory (see, for example, \cite{wall}). Two germs $f_i : (\K^n,0) \to (\K^p,0), i=1,2,$ and $\K=\C$ or $\R$, are said to be $\mathcal{R}$-\textit{equivalent} if there exists a germ of a diffeomorphism $ \phi:(\K^n,0) \to (\K^n,0)$ such that $f_2 = f_1 \circ \phi^{-1}$. 
 They are $\mathcal C$-\textit{equivalent} if there exists a smooth family of invertible $p\times p$ matrices $A : (\K^n,0) \to GL(p,\K)$ such that $f_2(x)=A(x)f_1(x)$.  
 They are $\mathcal K$-\textit{equivalent} if there exist $\phi$ and $A$ as above such that $f_2 (x)=A(x)(f_1 \circ \phi^{-1})(x)$.

Two germs $f_i : (\K^n,0) \to (\K^p,0)$, with $i=1,2$, are  said to be ${\mathcal A}$-\textit{equivalent} if there exist germs of diffeomorphisms $\phi:(\K^n,0) \to (\K^n,0)$ and $ \psi:(\K^p,0) \to (\K^p,0)$ such that 
$f_2=\psi\circ f_1 \circ \phi^{-1}$.

If ${\cal G}$ is one of the above groups ($\mathcal R,\mathcal K,\mathcal A$), a germ $f$ is said to be $k$-${\cal G}$-{\it determined} if any germ $g$ with $j^kg=j^kf$ is ${\cal  G}$-equivalent to $f$,  where $j^kf$ denotes the $k$-jet of $f$ at the origin (i.e., the Taylor polynomial of degree $k$, without the constant term).   A germ $f$ is said to be ${\cal G}$-{\it finite} if it is $k$-${\cal  G}$-determined for some $k$. There are algebraic criteria for determining whether a germ is ${\cal G}$-finite; see, for example,  \cite{wall}. 

\smallskip
In what follows, we consider germs of plane curves $(C,0)\subset (\K^2,0)$, where the curves may be defined either by parametrisations or by equations.

A germ of a parametrised curve $\gamma:(\C,0)\to (\C^2,0)$ is ${\cal A}$-finite if and only if $\gamma$ is an injective immersion away from the origin. Any such map can be written in the form $(t^m,y(t))$, where $\ord(y)>m$. Among the powers appearing in the series expansion of $y(t)$, the lowest exponent not divisible by $m$ is called the \textit{first Puiseux exponent} of $\gamma$, denoted $\beta(\gamma)$. If $\gamma$ is singular and ${\cal A}$-finite, then the first Puiseux exponent is finite (indeed, since 
$m>1$, if all exponents in 
$y(t)$ were divisible by 
$m$, then 
$y(t)=y_1(t^m)$ for some 
$y_1$, and $\gamma$ would fail to be injective). Clearly, $\beta(\gamma)$ is an analytic invariant of the curve $\gamma$.

Let $\mathcal{O}_s$ denote the ring of germs of holomorphic functions $(\C^s,0)\to \C$.  
A germ $f\in \mathcal{O}_2$ is ${\cal K}$-finite if and only if $f$ has an isolated singularity at the origin, which is equivalent to $f$ having no repeated factors. 
Such a germ decomposes into a product of irreducible factors in $\mathcal{O}_2$, namely $f = f_1\cdots f_n$. Clearly, the associated curves satisfy $C_f = C_{f_1} \cup \cdots \cup C_{f_n}$, where each $C_{f_i}$ is referred to as a \textit{branch} or \textit{component} of $C_f$.  

If $\gamma_j, j=1,2,$ are ${\cal A}$-finite germs  with reduced defining equations $f_j=0$, then  $\gamma_1$ and $\gamma_2$ are ${\cal A}$-equivalent if and only if $f_1$ and $f_2$ are ${\cal K}$-equivalent; see \cite{BruceGaffney}.  

If we write $f= F_r+F_{r+1}+\cdots,$
where each $F_i$ is a homogeneous polynomial of degree $i$ in the variables $x$ and $y$, and $F_r \neq 0$, then $r$ is the \textit{multiplicity} of $f$, denoted $\mult(f)$, and $C_{ F_r}$ is the \textit{tangent cone} of $C_f$.

The order of $\rho=\sum a_kt^k \in  \mathcal{O}_1$,  denoted by $\ord(\rho)$, is the smallest positive integer $k > 0$ such that $a_k\ne 0$. Equivalently,
$$
\ord(\rho) = \dim_\mathbb{\C} \frac{ \mathcal{O}_1}{\mathcal{O}_1\langle \rho \rangle},
$$
where $\mathcal{O}_1\langle \rho \rangle$ denotes the ideal in $ \mathcal{O}_1$ generated by $\rho$. Clearly, for any $\rho,\sigma\in \mathcal O_1$, we have  $\ord(\rho \sigma) = \ord(\rho)+\ord(\sigma)$. Moreover, if $\rho(0)=0$, then $\ord(\rho')=\ord(\rho)-1$.

 
\subsection{Vertices and inflections of regular curves}
Recall that two regular curves in $\C^2$, one parametrised by $\gamma: U\to \mathbb \C^2$ and  the other given by the equation $g=0,$ with $g:V\to \C$, where $U$ and $V$ are open subset of $\C$ and $ \C^2$ respectively, have {\it $k$-point contact} ($k\ge 1$) at $t_0$ if the following conditions are satisfied: 
$\gamma(t_0)\in V$, $(g\circ \gamma)(t_0)=0$, $(g\circ\gamma )^{(i)}(t_0)=0$, for $i=1,\ldots,k-1,$ and $(g\circ\gamma )^{(k)}(t_0)\ne 0$. 
Switching the roles of the two curves gives the same integer for the contact. The curves are said to have {\it at least 
$k$-point contact} if we drop the last condition.

We are interested in the local contact between a plane curve and lines and circles. For this purpose, it is convenient to both projectivise and complexify. A complex circle in $\C^2$ is defined as a conic in  $\BP\C^2$ with an equation of the form 
$$
A(x^2+y^2)+2Bxz+2Cyz+Dz^2=0.
$$

Such a conic passes through the circular points at infinity, $(1:\pm i:0)$. Its centre and radius are defined analogously to the real case, with the degenerate case  $A=0$ corresponding to lines.

Given a point $p\in\BP\C^2$ not on the line joining the circular points at infinity, we may choose a projective change of coordinates that maps $p$ to $(0:0:1)$, and then work in the affine chart with $z=1$. In this chart, the definitions of  inflection and vertex, dual and evolute and results on local structures for smooth complex curves proceed in the usual way; see \cite{BruceGiblin}. 
In the real case ($\K=\R$), the group preserving the circles is the group of similarities. In the complex case, if we projectivise, it is the group preserving the circular points at infinity. 
We may think of the subgroup fixing those points as analogous to the orientation preserving transformations, while the coset that exchanges them corresponds to orientation reversing transformations. 
In the affine chart $z=1$, one can check that the group fixing both the circular points at infinity and the origin consists of the linear maps of the form $(x,y)\mapsto (ax+by,-bx+ay),$ with $a, b\in \C, a^2+b^2\ne 0$.

\subsection{Multiplicity and intersection number}\label{subs:MultInter}

Given $f, g \in \mathcal{O}_2$ the  \textit{intersection number} of $f$ and $g$, denoted by $m(f,g)$, is the codimension of the ideal $\mathcal{O}_2\langle f,g \rangle$ in $\mathcal{O}_2$. That is, 
 $$
 m(f,g) = \dim_{\C}  \frac{\mathcal{O}_2}{\mathcal{O}_2\langle f,g \rangle}.
 $$ 

Further details about the intersection multiplicity (or the intersection index) can be found in \cite{Hefez}. We will use the following properties:
\begin{center}
\begin{tabular}{ll}
	1. $m(f,g)=m(g,f)$;
	&
		4. $m(f,g+fh)=m(f,g)$;
		\\	
	2. $m(uf,vg)=m(f,g)$\, \mbox{\rm (for units $u,v \in \mathcal{O}_2$)};
	&
	5. $m(f \circ \Phi,g \circ \Phi) = m(f,g)$.
	\\
	3. $m(f,gh) = m(f,g)+m(f,h)$;
	&
\end{tabular}
\end{center}
Here, $f,g,h \in \mathcal{O}_2$, and $\Phi:(\C^2,0)\to (\C^2,0)$ is a bi-holomorphic map-germ. 

Note that if $f$ and $g$ are irreducible, then $m(f,g)$ is non-finite if and only if $f$ and $g$ define the same curve. The above properties follow directly from the definition.  However, what is not simple to prove is that $m(f,g)$ is the local number of inverses images of a regular value of $(f,g):(\C^2,0)\to (\C^2,0)$; see \cite{Arnold}, Chapter 5. While the following results are standard, we outline their proofs to keep the presentation relatively self-contained.

\begin{prop}\label{prop:multOrderParm}
{\rm (1)} Let  $\gamma:(\C,0)\to (\C^2,0)$ be a parametrisation of  a germ  of an irreducible curve $C_f$. Then $m(f,g)=\mbox{\rm \ord} (g\circ \gamma)$ {\rm (\cite{Hefez})}. If $f=0$ is reducible with $r$ branches parametrised by $\gamma_j$, then 
$m(f,g)=\sum_{j=1}^r \mbox{\rm \ord}(g\circ \gamma_j).$

{\rm (2)} Suppose that $C_f$ is irreducible with parametrisation $\gamma(t)=(t^m,y(t))$, where $\mbox{\rm \ord}(y(t))>m$. Then one can write the equation of the curve in the form $y^m+g(x,y)=0$, where $g\in{\cal O}_2^{m+1}$.
	
{\rm (3)} If $f$ and $g$ are irreducible and their tangent cones are transverse, then $m(f,g)=\mbox{\rm \mult} (f)\, \mbox{\rm \mult}(g)$.

{\rm (4)} If $f=0$ is irreducible, then $m(f,g)=\infty$ if and only if $f=0$ is a component of $g=0$.
\end{prop}

\begin{proof}
(1) It is enough to prove the result when $f=0$ is irreducible, by property (3) of the multiplicity stated earlier. So suppose that we have a parametrisation $\gamma:(\C,0)\to (\C^2,0)$ of $f=0$. Consider the function $G=g\circ \gamma:(\C,0)\to(\C,0)$ and let $c$ be a regular value of $G$ near $0\in\C$. Then the  set $G^{-1}(c)$ will (locally) consist of $\ord (G)$ points. 
On the other hand, $t\in G^{-1}(c)$ if and only if $p=\gamma(t)\in g^{-1}(c)\cap f^{-1}(0)$. 
It remains to show that $(0,c)$ is a regular value of $(f,g):(\C^2,0)\to (\C^2,0)$. If it were not, then for some $t\in G^{-1}(c)$ with $p=\gamma(t)$, the differentials $df(p)$ and $dg(p)$ would be linearly dependent, implying that $t$ is a critical point of $G$, contradicting the regularity of $c$.

(2) We may write 
$$y(t)=\sum_{j=0}^{m-1}t^j\alpha_j(t^m)=\sum_{j=0}^{m-1}t^j\alpha_j(x).$$ 
Consider $t^ry(t), 0\le r\le m-1$, replacing $t^m$ by $x$ at each stage. 
Rearranging, we obtain $A(x,y)T=0$, where $A$ is an $m\times m$ matrix whose entries are holomorphic functions in $(x,y)$ and $T$ is the transpose of $(1,t,\ldots,t^{m-1})$. Taking the determinant gives the equation for $\gamma(t)$ with the required property. 
	
(3) By a diffeomorphism, we may suppose that the tangent cones of $f$ and $g$ are given by 
$x=0$ and $y=0$, respectively. Then we can parametrise $g=0$ as $\gamma(t)=(t^{m_1},y(t))$.
By part (2), we may write $f(x,y)=x^{m_2}+f_1(x,y)$, where  both $y(t)$ and $ f_1(x,y)$ have orders  greater than $ m_1$ and $m_2$, respectively. Then $\ord (f\circ \gamma) = m_1m_2$, as required.

(4) Parametrising $f=0$ by $\gamma$, we have $m(f,g)=\infty$ if and only if $g\circ \gamma$ vanishes identically, i.e., if and only if $f=0$ is a component of $g=0$.
\end{proof}


\subsection{Contact between singular and smooth curves}
Suppose we are given the germ of a plane curve $(C_\gamma,0)$ realised by a parametrisation $\gamma:(\C,0)\to (\C^2,0)$. Given any germ of a submersion $h:(\C^2,0)\to (\C,0)$, we consider the order of $h\circ \gamma$. Varying over all such submersions, we obtain a set of natural numbers, which we denote by ${\cal I}(\gamma)$.

Similarly, suppose we are given a plane curve $(C_f,0)$ defined by the equation $f=0$, where 
$f:(\C^2,0)\to (\C,0)$. Given any germ of an immersion $\alpha:(\C,0)\to (\C^2,0)$, we consider the order of $f\circ \alpha$. Varying over all such immersions, we again obtain a set of natural numbers, which we denote by   ${\cal I}(f)$. 

We could naturally extend the notion of ${\cal I}(\gamma)$ to multi-germs $\gamma:(\C,S)\to (\C^2,0)$, where $S\subset \C$ is a finite set, in order to maintain the symmetry between the two approaches. We leave the details to the reader. 

The results that follow are related to those in \cite{Porteous}, where the maps $\alpha$ and $h$ are referred to as {\it probes}. Clearly, if $\gamma$ is an immersion and $f$ a submersion, then ${\cal I}(\gamma)={\cal I}(f)=\N$.

\begin{theo}\label{theo:probs}
{\rm (1)} ${\cal I}(\gamma)$ is an ${\cal A}$-invariant of $\gamma$, and ${\cal I}(f)$ is a ${\cal K}$-invariant of $f$.

{\rm (2)} If $(C_f,0)$ and $(C_\gamma,0)$ define the same curve, then ${\cal I}(\gamma)={\cal I}(f)$.

{\rm (3)} If $f$ is irreducible and singular, then 
$$
{\cal I}(f)=\{km:  1\le k<\frac{\beta(f)}{m}\}\cup\{\beta(f)\},
$$ 
where $m=\mbox{\rm \mult}(f)$, and $\beta(f)$ is the first Puiseux characteristic exponent of $f$. In particular, ${\cal I}(f)=\{ m, \beta(f)\}$ has exactly two elements if and only if $\beta(f)<m$, with $m$ occurring when the the curve representing the probe is transverse to the tangent cone of $f$,  and $\beta(f)$ occurring when it is not.

{\rm (4)} If $f$ and $g$ are irreducible and their tangent cones are transverse, then  
$$
{\cal I}(fg)= \left(\mbox{\rm \mult}(f)+{\cal I}(g)\right) \cup \left( \mbox{\rm \mult}(g)+{\cal I}(f) \right).
$$

{\rm (5)} If $f_i, i=1,\ldots,k,$ are irreducible, then  
$${\cal I}(f_1 \cdots  f_k)\subset  {\cal I}(f_1)\cup\cdots\cup{\cal I}(f_k).$$ 
In particular, if the $f_i$ are irreducible and singular, then  ${\cal I}(f_1\cdots f_k)$ is finite. 

{\rm (6)} The set ${\cal I}(f)$ is unbounded if and only if $f=0$ contains a smooth component. 

{\rm (7)} 
Suppose $f:(\C^2,0)\to (\C,0)$ is a ${\cal K}$-finite germ with decomposition into irreducible factors $f=\Pi_{1\le j\le k, 1\le i\le r_j}f_{ij}$, where for each $j$,  $f_{ij}, i=1,\ldots, r_j,$ share the same tangent direction $L_j$, and the directions $L_j$ are pairwise distinct. Assume also that that $\beta(f_{ij})<\mbox{\rm \mult}(f_{ij})=m_{ij}$ for all $i, j$. Let $M_j=\sum_{i=1}^{r_j} m_{ij}$,  $\beta_j=\sum_{i=1}^{r_j}\beta(f_{ij})$ and $M=\sum_{j=1}^{s} M_j$. 
Then 
$$
{\cal I}(f)=\{M, M-M_1+\beta_1,\ldots,M-M_k+\beta_k\}.
$$ 
\end{theo} 

\begin{proof}
(1), (2) are straightforward. 

For (3), suppose that $\gamma(t)=(t^m,y(t))$ with $\ord(y(t))>m$. Write $f=ax+by+O(2)$.
Then $\ord(f\circ \gamma)=m$ if $a\ne 0$. If $a=0$ then $b\ne 0$, and by a change of coordinates of the form $(x,y)\mapsto (x,y/b+A(x,y)),$ where $A\in {\cal M}_2^2$,  we can reduce $f$ to the form  $y$, 
while $\gamma$ remains in the same form $(t^m,y(t))$. Thus, ${\cal I}(\gamma)$ consists of the set $\{\ord(y(t))\}$,  as $(t^m,y(t))$ ranges over all parametrisations ${\cal A}$-equivalent to $\gamma$.  The result now follows.

For (4), if the tangent cones are transverse, then we can write the product as $(x^r+F(x,y))(y^s+G(x,y))$, where $F\in {\cal M}^{r+1}, G\in {\cal M}^{s+1},$ and  $r=\mult(f), s=\mult(g)$. Any germ of an immersion can be written in the form $(t,y(t))$ or $(x(t),t)$. In the former case,  substitution gives an expression with order $r+\ord(g\circ \gamma)$, and in the latter case, $s+\ord(f\circ \gamma)$. 

For (5), clearly $\ord((f_1\cdots f_k)\circ \gamma)=\sum_{i=1}^k \ord(f_i\circ \gamma)$. 

For (6), if all the irreducible components of $f=0$ are singular, then ${\cal I}(f)$ is finite by (5). 
On the other hand, if $f=0$ has a smooth component, say $y=0$, consider the immersion $\alpha(t)=(t,t^n)$. Clearly, $\ord(f\circ \alpha)=n$, so ${\cal I}(f)$ is unbounded.   

Finally, for (7), there are only $k+1$ relevant tangent directions for probes $\alpha$ from (3), resulting in the indicated orders. 
\end{proof}

\begin{exs}
{\rm
The $\mathcal K$-simple singularities of germs of functions were classified by Arnold. When $n = 2$, they are $\mathcal{K}$-equivalent  (over $\K=\C$) to the following normal forms:
$$\begin{array}{ll}
	A_k:&x^2 + y^{k+1}, k \geq 1, \\ 
	D_k:& x^2 y+y^{k-1}, k \geq 4, \\
	E_6:& x^3+y^4, \\
	E_7:&x^3+xy^3, \\
	E_8:&x^3+y^5.
\end{array}
$$

We now compute ${\cal I}(f)$ for each of these singularities using Theorem~\ref{theo:probs}.

(1) $A_{2k}$-singularity: there is a single branch, and the Puiseux exponent is $\beta(f)=2k+1$, while the multiplicity of $f$ is equal to $2$. By Theorem \ref{theo:probs}(3), we have:
 $${\cal I}(A_{2k})=\{2, 4,\ldots, 2k, 2k+1\}.$$ 

(2) $A_{2k+1}$-singularity: for the case $A_1$, one can check that  
${\cal I}(A_1)=\{l\in \mathbb N: l\ge 2\}$. 

For $k\ge 2$, using the normal form $x^2-y^{2k}=(x-y^k)(x+y^k)$, so the curve has two smooth branches. The probes $\alpha(t)$ can be written in the form $(t,y(t))$ or $(x(t),t)$, with contact respectively $(t-y(t)^k)(t+y(t)^k), (x(t)-t^k)(x(t)+t^k)$.  In the first case, the order of  $f\circ\alpha$ is $2$, and from the second case, we see that 
$$
{\cal I}(A_{2k+1})=\{2, \ldots, 2k, 2k+j:  j\ge 2\}.
$$ 

(3) $D_k$-singularity: we use the form  $f=y(x^2+y^{k-2})$, which consists of two branches: a smooth branch ($A_0$) 
and a singular branch of type $A_{k-3}$. The branches have transverse  tangent cones, so by Theorem \ref{theo:probs} (4), we have
$$
{\cal I}(D_k)=\left(1+{\cal I}(A_{k-3})\right) \cup \left(2+{\cal I}(A_0)\right) =\{l\in \mathbb N, l\ge 3\}.
$$ 

(4) $E_6, E_7, E_8$: it is not hard to show that  
$$
{\cal I}(E_6)=\{3, 4\}, \quad {\cal I}(E_7)=\{3\}\cup\{j\in \mathbb N: j\ge 5\},\quad  {\cal I}(E_8)=\{3,4, 5\}.
$$ 
}
\end{exs}


\section{Geometric invariants of singular plane curves} \label{sec:vert}

The fundamental question we address is: when the germ of a plane curve is perturbed, how many inflections and vertices emerge? As in earlier discussions, it is crucial to distinguish whether the curve is defined by a parametrisation or by an implicit equation, since the generic deformations in these two cases differ—even topologically.

The parametrised case has been studied in detail in \cite{DiasTari} and \cite{wallflat}. Here, we examine both approaches in parallel.

Given a germ of an ${\mathcal A}$-finite map $\gamma:(\C,0)\to(\C^2,0)$, or a ${\cal K}$-finite function $f:(\C^2,0)\to (\C,0)$, we can mimic the usual definition of curvature and its derivative at regular points of $\gamma(\C)$ and $C_f$. 
 
For vertices, we consider contact with circles passing through the origin. Such circles are given by an equation $A(x^2+y^2)+2Bx+2Cy=0$.  If the curve is parametrised, then the origin is a vertex if the contact of the curve with a circle at 
$ t=0$ is  of order $\ge 4$, equivalently, 
$$
\begin{pmatrix}
(xx'+yy') & x' & y'\\
(xx''+yy''+x'^2+y'^2) & x'' & y''\\
(xx'''+yy''') +3(x'x''+y'y'') & x''' & y'''
\end{pmatrix}
\begin{pmatrix}
A\\B\\C
\end{pmatrix}
\begin{matrix}
\\=\\
\end{matrix}
\begin{pmatrix}
0\\0\\0
\end{pmatrix}.
$$

(Note that this includes the lines as degenerate circles $A=0$.)
We obtain a linear system in the variables $A,B,C$, which admits a non-trivial solution if and only if the (Wronskian) determinant vanishes, that is, 
$$v_{\gamma}=(x'^2 + y'^2)(x'y''' -x'''y') + 3(x'x''+y'y'')(x''y'-x'y'')=0.$$ 

Note that when $A=0$, the equation $A(x^2+y^2)+2Bx+2Cy=0$ reduces to the family of lines through the origin. In this case, the first two equations of the linear system involve only the unknowns $B$ and $C$, and they have a non-trivial solution if and only if the determinant
 $i_{\gamma}=x'y''-x''y'$ vanishes at $t=0$.

The corresponding conditions for inflections and vertices in the case of curves $C_f$ 
defined implicitly by $f=0$ are obtained in the smooth case from the implicit function theorem. 

\begin{defn} \label{def:IfVf}

{\rm (1)} We say that the origin  is an {\it inflection of $C_\gamma$} 
when 
$i_\gamma(0)=0$, where 
$$i_\gamma(t)=x'(t)y''(t)-y'(t)x''(t).$$

We say that the origin is an {\it inflection of  $C_f$} when $(x,y)=(0,0)$ is a solution of the equations
\begin{equation}\label{eq:sis_I}
	\left\{\begin{array}{l}
		f(x,y) = 0,\\
		i_f(x,y) = 0,
	\end{array}\right.
\end{equation}
where 
$$
i_f = f_y^2f_{xx} -2  f_x f_yf_{xy}+ f_x^2f_{yy}.
$$ 

{\rm (2)} We say that the origin is a vertex of $C_\gamma$ when  $v_\gamma(0)=0$, where $$v_\gamma(t)=((x'^2+y'^2)(x'y''-x''y')+3(x'x''+y'y'')(x''y'-x'y'')).$$ 

We say that the origin is a vertex of $C_f$ when $(x,y)=(0,0)$ is a solution of the equations
\begin{equation}\label{eq:sis_V}
\left\{\begin{array}{l}
f(x,y) = 0,\\
v_f(x,y) = 0,
\end{array}\right.
\end{equation}
where
$$
\begin{array}{rl}
	v_f = & (f_x^2+f_y^2) \left( f_{x}^3 f_{yyy}-3 f_{x}^2 f_{y} f_{xyy} +3 f_{x} f_{y}^2 f_{xxy} -f_{y}^3 f_{xxx} \right)\\
	& -3((f_x^2- f_y^2)f_{xy} - f_xf_y(f_{xx} -f_{yy}) )(f_y^2f_{xx} -2  f_x f_yf_{xy}+ f_x^2f_{yy}).
\end{array}
$$

{\rm (3)} We 
define the number of inflections $I_\gamma$ and vertices $V_\gamma$ of a parametrised curve $C_\gamma$, concentrated at the origin, by
$$
I_{\gamma}=\ord\ i_\gamma \quad \textrm{and} \quad V_{\gamma}=\ord\ v_\gamma.
$$

For curves $C_f$ given implicitly by $f=0$, we 
define the number of inflections $I_f$ and vertices $V_f$ of $C_f$, concentrated at the origin, by
$$
I_f = m(f,i_f) \quad \textrm{and} \quad V_f = m(f,v_f).
$$

{\rm (4)}  We say that an inflection $($resp. vertex$)$ is \textit{ordinary/simple} when $I_f = 1$ $($resp. $V_f = 1)$. 

\end{defn}

\begin{rems}
{\rm 
(1) The integers $I_\gamma$ and $  V_\gamma$ were defined in \cite{DiasTari}; note that $I_\gamma, I_f$ and $V_\gamma, V_f$   do not  generally coincide when $f$ is irreducible and $C_f$ and $ C_\gamma$ define the same germ.

(2) If we have a homogeneous polynomial $F$ such that $F(0:0:1)=0$, let $f(x,y)=F(x,y,1)$ be its affine form. Then the intersection number of the Hessian of $F$, restricted to the affine chart $z=1$, with $F=0$ at the point $(0:0:1)$ is equal to $I_f$.   
}
\end{rems}

\begin{theo}\label{theo:f=gh}
{\rm (1)}  The integers $I_{\gamma}$ and $I_f$ are affine invariants, while $V_\gamma$ and $ V_f$ are invariant under similarities, interpreted as above in the complex case. An inflection {\rm(}resp. vertex{\rm)} is ordinary if and only if $f$ is a submersion at the origin and the origin is an ordinary inflection {\rm(}resp. vertex{\rm)}, that is, the curve has $3$-point $($resp. $4$-point$)$ contact with its tangent line $($resp. osculating circle$)$. 

{\rm (2)} At smooth points of $C_f$, {\rm Definition \ref{def:IfVf}(4)} corresponds to the classical notions of inflection and vertex. Moreover, $I_f+2$ is the order of contact between $C_f$ and its tangent line at the origin.  If the centre of curvature is finite, then $V_f$ is the order of contact between $C_f$ and its osculating circle; otherwise, the point is an inflection, and $V_f= I_f-1$.  

{\rm (3)} Any singular point of $C_f$ is, by {\rm Definition \ref{def:IfVf}}, both an inflection and a vertex.

{\rm (4)} The integers $I_\gamma$ and $ V_\gamma$ $($resp. $I_f$ and $V_f)$, when finite, depend only on a finite jet of $\gamma$ $($resp. $f)$.

{\rm (5)} If $f=gh$, and $I_g$ and $I_h$ are finite, then $I_f$ is finite and 
$$
I_f=I_g+I_h+6m(g,h).
$$
Similarly, if $V_g$ and $V_h$ are finite, then $V_f$ is finite and 
 $$
 V_f=V_g+V_h+12 m(g,h).
 $$
 
{\rm (6)} The integers $I_{\gamma}$ and $V_{\gamma}$ are ${\mathcal R}$-invariants, while $I_f$ and $ V_f$ are ${\cal C}$-invariants.

{\rm (7)}  The invariant $I_\gamma$ $($resp. $I_f)$ is infinite if and only if $(C_\gamma,0)$ $($resp. $(C_f,0))$ contains the germ of a line. Likewise,  $V_\gamma$ $($resp. $V_f)$ is infinite if and only if $(C_\gamma,0)$ $($resp. $(C_f,0))$  contains the germ of a complex circle or line.
\end{theo}

\begin{proof}
 Statements (1)–(3) are straightforward. For (4), we provide a proof only in the case of inflections; the vertex case follows similarly.
 
 Let ${\cal M}_2$ denote the maximal ideal in ${\cal O}_2$. The regular case is similar to the singular case, but in the singular case we can prove a stronger statement: if $f$ is singular and $I_f=k$, then $I_{f+h}=I_f$ for any $h\in{\cal M}^{k+1}$. 
 
Assume $f$ is singular and $I_f=k$.  By Nakayama's Lemma (\cite[Lemma 1.4 (ii)]{wall}),  we have  ${\cal M}_2^{k}\subset {\cal O}_2\langle f,i_f\rangle$. 
Now replace $f$ by $f+h$, where $h\in {\cal M}_2^{k+1}$. Since $f$ is singular, we have $i_{f+h}=i_f+H$, with  $H\in {\cal M}_2^{k+1}$. It follows that
$$
{\cal M}_2^k\subset {\cal O}_2\langle f,i_f\rangle 
\subset {\cal O}_2\langle f+h,i_{f+h}\rangle+ {\cal M}_2^{k+1}.
$$

Again applying Nakayama's Lemma, we conclude that  
${\cal M}_2^{k}\subset {\cal O}_2\langle f+h,i_{f+h}\rangle$,
which implies 
${\cal O}_2\langle f,i_f\rangle= {\cal O}_2\langle f+h,i_{f+h}\rangle$, 
and so $I_f=I_{f+h}$.

Now for the additivity statement in (5): we observe that  $i_f = g^3 i_h+h^3 i_g+g h r_1,$ for some $r_1$. Using properties of the multiplicity, we compute:
	$$\begin{array}{rcl}
		I_f & = & m(gh,g^3 i_{h}+h^3 i_{g}+g h r_1)\\
		 &=&m(g,h^3 i_{g})+m(h,g^3 i_{h})\\
		    & = & m(g,i_g)+m(h,i_h)+m(g,h^3)+m(h,g^3)\\
		    &=&I_g+I_h+6m(g,h).
	\end{array}
$$

Similarly, a direct computation shows that 
$v_f= g^6v_h +h^6v_g+ gh r_2$ for some $r_2$, and the the formula for $V_f$ follows in the same way.

Statement (6) is immediate from the definitions.

For (7), suppose that $I_f$ is infinite. Then $C_f$ and the zero set of $i_f$ share a component. At any smooth point 
$p$ of this shared component,  $I_f$ is infinite, and by (2) this implies that the curve $C_f$ has infinite contact with its tangent line — i.e., locally it contains a line. The same argument yields the other cases.
 \end{proof}
 
By applying induction to Theorem~\ref{theo:f=gh}(5), we arrive at the following result.

\begin{theo}\label{teo:calc_fat_irred}
	If $f=f_1 \dots f_n$, then:
$$
\begin{array}{c}
	I_f = \displaystyle{\sum_{i}^{n} I_{f_i}+6\sum_{i<j} m(f_i,f_j);}\\[0.2cm]
	V_f = \displaystyle{\sum_{i}^{n} V_{f_i}+12\sum_{i<j} m(f_i,f_j).}
\end{array}
$$
\end{theo}

One question is: what is the link between the invariant  $I_\gamma$ (resp. $V_\gamma$) and the order of contact between the curve and its tangent lines (resp. circles)? One might expect that the higher the order of contact between $C_\gamma$ and its tangent line the larger $I_\gamma$ might be. 
 
\begin{theo}\label{paramcase}
{\rm (1)} Let $\gamma:(\C,0)\to (\C^2,0)$. There is a unique osculating complex circle or line, that is one with maximal order of contact with $C_{\gamma}$ at the origin. If $\gamma$ is singular, this order of contact is finite. The maximal order of contact with a circle is denoted it by $\lambda(\gamma)$.

{\rm (2)} If $\gamma (t)=(t^m,at^n+O(n+1)), n>m$, then the order of contact between $C_\gamma$ and its tangent line is $n$ and $I_{\gamma}=m+n-3$.

{\rm (3)} $V_{\gamma}=I_{\gamma}+\lambda(\gamma)-3$.

{\rm (4)} Let $\gamma:(\C,0)\to (\C^2,0)$ be as above. 
\begin{itemize}
	\item[] If $n\ne 2m$, then $V_{\gamma}=3m+n-6$.
	\item[] If $n=2m$, then writing $y(t)=t^{2m}(a_{2m}+y_1(t)), y_1\in{\cal M}_1$ we have:
	$$
	V_{\gamma}=3m+n+\ord(a_{2m}t^{2m}(a_{2m}+y_1(t))^2-y_1(t))-3.
	$$
\end{itemize} 
\end{theo}

\begin{proof}
(1) We claim that if $C_\gamma$ is singular, there is unique circle or line with maximum finite order of contact, that is, there is a unique osculating circle if we consider lines as circles through infinity.  The circles and lines through the origin are given by the equation  $A(x^2+y^2)+2Bx+2Cy=0$, where not all of $A, B, C$ are zero. 
In the singular case, by a rotation we may suppose that $x(t)=t^m, y(t)=t^nY(t),$ with $ Y(0)\ne 0$ and  $n>m$, and set
 $$
 d(t)=A(t^{2m}+t^{2n}Y(t)^2)+2Bt^m+2Ct^ny_1(t).
 $$ 
 
This has order $m$ unless $B=0$, so the osculating circle/line lies in the pencil given by $B=0$. 
\begin{itemize}
\item 
If $2m<n$, then the order is $2m$ unless $A=0$; so the unique osculating \lq circle\rq\ in this case is the line $y=0$, and the order of contact is $n$. All the tangent circles have contact of order $2m$. 
\item 
If $n<2m$, the order is $n$ unless $C=0$,  in which case the osculating circle is $x^2+y^2=0$, and the order is $2m$. 
\item 
If $n=2m$, setting $A=1$, write $y(t)=t^{2m}(a_{2m}+y_1(t)),$ with $ y_1\in{\cal O}_1$, and take $C=1/2a_0$. Then  the osculating circle is $x^2+y^2-y/a_0=0$, and the order of contact is $2m+\ord(a_{2m}t^{2m}(a_{2m}+y_1(t))^2-y_1(t))$. 
\end{itemize}

Note that the expression $a_{2m}t^{2m}(a_{2m}+y_1(t))^2-y_1(t))$ cannot vanish identically, since this would imply that $y_1(t)=y_2(t^m)$ for some $y_2$, contradicting that the parametrisation is reduced. Of course, it is classical that when $C_{\gamma}$ is smooth, there is a unique osculating circle or line, but in that case the order of contact may be infinite.

(2) For inflections, the lowest order terms appearing in the Wronskian are
$$
\begin{vmatrix}
mt^{m-1}& nt^{n-1}\\
m(m-1)t^{m-2} & n(n-1)t^{n-2}
\end{vmatrix}=mn(n-m)t^{m+n-3},
$$
 so $I_{\gamma}=m+n-3$.

(3) Suppose that the order of contact with the circle or line   $A(x^2+y^2)+2Bx+2Cy=0$ is $k=\lambda(\gamma)$, so $A(x(t)^2+y(t)^2)+2Bx(t)+2Cy(t)=2dt^k+O(k+1)$ for some $d\ne 0$. We then have
{\small 
	$$
\begin{pmatrix}
(xx'+yy') & x' & y'\\
(xx''+yy''+x'^2+y'^2) & x'' & y''\\
(xx'''+yy''') +3(x'x''+y'y'') & x''' & y'''
\end{pmatrix}
\begin{pmatrix}
A\\B\\C
\end{pmatrix}
\begin{matrix}
\\=\\
\end{matrix}
\begin{pmatrix}
dkt^{k-1}+O(k)\\dk(k-1)t^{k-2}+O(k-1)\\dk(k-1)(k-2)t^{k-3}+O(k-2)
\end{pmatrix}.
$$
}
Let us denote this as $M(t)D=\alpha(t)$. The adjugate of the matrix $M(t)$ has as its first row the entries $(x''y'''-x'''y'',x'''y'-x'y''',x'y''-x''y')$, whose leading (non-vanishing) terms are 
$$
(m(m-1)n(n-1)(n-m)at^{m+n-5}, mn(m-n)(m+n-3)at^{m+n-4},mn(n-m)at^{m+n-3}).
$$

Now $adj(M(t))M(t)=\det M(t)I$, and the order of $\det M(t)$ is $V_{\gamma}$. The entry of $adj(M(t))\alpha(t)$ of lowest degree is   $t^{m+n+k-6}$ with coefficient
$$
adkmn(n-m)(k-m)(k-n).
$$ 

We need $k\ne m, n$. 
\begin{itemize}
	\item If $n<2m$, we must have $B=C=0$ and $k=2m$. 
	\item If $n>2m$, then $B=0$ and $k=2m$. 
	\item If $n=2m$, then $B=0$ and we must choose $A, C$ so that  $d(t)$ has order greater than $2m$, that is, we have the osculating circle.
\end{itemize} 

In all cases $k=\lambda(\gamma)$, and comparing orders we have:
$m+n+k-6=V_{\gamma}$, so that $V_{\gamma}=I_{\gamma}+\lambda(\gamma)-3$.

(4) Follows immediately from (1), (2) and (3).
\end{proof}

\begin{rems}\label{rem:finiteIV}
{\rm
(1) In  Theorem \ref{paramcase}, we see that the invariants $I_\gamma$ and $V_\gamma$ can be expressed in terms of the multiplicity of $\gamma$ and the 
the order of contact with its tangent line, and with its osculating circle or line. Except in the case when that osculating circle is the degenerate one given by $x^2+y^2=0$, these circles and lines are smooth, and therefore  the orders of contact must lie in ${\cal I}(\gamma)$. 

(2) Given an ${\mathcal A}$-finite germ $\gamma:(\C,0)\to (\C^2,0)$ (resp. a ${\mathcal K}$-finite germ $f:(\C^2,0)\to (\C,0)$), there is an arbitrarily small deformation of $\gamma$ (resp. of $f$) such that the associated curve has only simple inflections and vertices.

(3) Except in the regular case, the curvature function $x'y''-y'x''$ is {\it never} versally unfolded by varying the parametrisation. Consequently, there is no immediate model describing the behaviour of simple inflections and vertices under deformation. For approaches to establishing such models, we refer the reader to \cite{DiattaGiblin, Garay, SalariTari, InfVertProfiles}.
}
\end{rems}

We have the following consequence of Theorems \ref{theo:probs}, \ref{theo:f=gh}, \ref{teo:calc_fat_irred}, and Remarks \ref{rem:finiteIV}.

\begin{cor}
Suppose that 
$f$ has no smooth components. Then the set ${\cal I}(f)$ is finite, and consequently both $I_f$ and $V_f$ are finite. 
\end{cor}

\section{Computing $I_f$ and $V_f$ and their ranges}\label{sec:irreducible}

By Theorem \ref{teo:calc_fat_irred}, to determine $I_f$ and $V_f$ for $f\in {\mathcal O}_2$, it is enough to compute these numbers for each irreducible component of $f$. So we assume that $f$ is irreducible with multiplicity $m$, and consider a parametrisation of $C_f$ in the form
\begin{equation}\label{eq:puiseux}
\gamma(t) = (t^m,y(t)),
\end{equation}
with $y(t) \in \mathcal{O}_1$ and $\ord(y) > m$, so that
\begin{equation}\label{eq:Hefez}
	m(f,g) = \dim_\mathbb{C} \frac{\mathcal{O}_2}{\langle f,g \rangle} = \dim_\mathbb{C} \frac{\mathcal{O}_1}{\langle g\circ \gamma \rangle} = \ord(g \circ \gamma),
\end{equation}
for any $g \in \mathcal{O}_2$ (see Proposition \ref{prop:multOrderParm}).

As above, the integers $I_\gamma$ and $V_\gamma$  represent the number of inflections and vertices that arise when deforming $\gamma$. If $\gamma$ is a parametrisation of $C_f$, then it is natural to expect that $I_\gamma \leq I_f$ and $V_\gamma \leq V_f$, since deformations of $\gamma$ arise from special deformations of $f$. The following result relates $I_\gamma$ and $I_f$, and $V_\gamma$ and $V_f$, where $\mu(f)=m(f_x,f_y)$ is the Milnor number of the germ $f$.

\begin{theo}\label{teo:princ}
Let $f : (\C^2,0) \to (\C,0)$ be a germ of an irreducible holomorphic function, 
and let $\gamma : (\C,0) \to (\C^2,0)$ be a parametrisation of the curve $C_f$. Then:
\begin{itemize}
\item[\rm (1)] {\rm (\cite[Proposition 4.1]{wallflat})} $I_f = I_\gamma + 3 \mu (f).$
\item[\rm (2)] $V_f = V_\gamma + 6 \mu (f).$
\item[\rm (3)] Let $f=f_1\ldots f_n$, where each $f_i$ is irreducible, and let $\gamma_i$ be a parametrisation of the corresponding branch $C_{f_i}$. Then: 
$$
\begin{array}{c}
	I_f=\displaystyle{\sum_{i=1}^n I_{\gamma_i}+3(\mu(f)+n-1),} \\[0.2cm]
	V_f=\displaystyle{\sum_{i=1}^n V_{\gamma_i}+6(\mu(f)+n-1).}
\end{array}
$$
\end{itemize}
\end{theo}

\begin{proof} 
For (1), we give an alternative proof to that in  \cite{wallflat}. 
If $\gamma(t)=\left(x(t),y(t)\right)$ as above, so that $f(x(t),y(t)) \equiv 0$. Then:
\begin{equation}\label{eq:der_F_t}
	\begin{aligned}
		&f_x x' + f_y y' = 0, \\
		&f_x x'' + f_y y'' + f_{xx}(x')^2 + 2f_{xy}x'y' + f_{yy}(y')^2 = 0, \\
		&f_x x'''+ f_y y'''+ f_{xxx}(x')^3 + 3f_{xxy}(x')^2 y' + 3f_{xyy}x'(y')^2 + f_{yyy}(y')^3 \\
		&\qquad + 3\left(f_{xx}x'x'' + f_{xy}(x''y' + x'y'') + f_{yy}y'y''\right) = 0.
	\end{aligned}
\end{equation}

Here the derivatives of $f$ are evaluated at $\gamma(t)$, and the derivatives of $x$ and $y$ are taken with respect to $t$. 

From the first equation, we deduce that there exists  $u: (\C,0) \to \C$, 
with $u(0) \neq 0$, and an integer $k \geq 0$, such that:
\begin{equation}\label{eq:lambda}
f_x(x(t),y(t)) = u(t) t^k y'(t) \quad \textrm{and} \quad f_y(x(t),y(t)) = -u(t) t^k x'(t).
\end{equation}

Now observe that $f(0,y)=y^m$ for $y\ne 0$ small, so by Teissier's Lemma (see \cite{Ploski}):
\begin{equation}\label{eq:teissier}
	m(f(x,y),f_y(x,y)) = \mu(f)+m(f(0,y),x)-1=\mu(f)+m-1.
\end{equation}
       
Using (\ref{eq:Hefez}), (\ref{eq:teissier}) and (\ref{eq:lambda}) we get:
$$
\ord(f_y(x(t),y(t))) = \mu(f)+m-1=k+m-1,
$$
hence $k=\mu(f)$.

Multiplying the second equation of (\ref{eq:der_F_t}) by $u(t)^2 t^{2k}$ and using (\ref{eq:lambda}), we find:
$$
\left(f_{xx} f_y^2-2 f_{xy} f_x f_y+f_{yy} f_x^2\right)(x(t),y(t)) = -u(t)^3 t^{3k} (x''(t) y'(t)-y''(t)x'(t)),
$$ 
that is, $i_f(\gamma(t)) = - u(t)^3 t^{3k} i_\gamma(t)$. Therefore,
\begin{equation}\label{eq:If}
\begin{array}{rcl}
I_f & = & m(f,i_f)= \ord(i_f \circ \gamma)=\ord(- u(t)^3 t^{3k} i_\gamma(t))\\
& = &  \ord (i_\gamma)+3k =I_\gamma+3k=I_{\gamma}+3\mu(f).
\end{array}
\end{equation}

Multiplying the third equation of (\ref{eq:der_F_t}) by $(f_x^2+f_y^2) u(t)^3 t^{3k}$, and using (\ref{eq:lambda}) and the second equation of (\ref{eq:der_F_t}), it follows that $v_f(\gamma(t)) = - u(t)^6 t^{6k} v_\gamma(t)$. 
Using calculations similar to those in (\ref{eq:If}), we conclude that
$V_f =V_\gamma+ 6 k=V_\gamma+ 6\mu(f)$.

Part (3) follows from parts (1) and (2) above, together with Theorem \ref{teo:calc_fat_irred} and Theorem 6.5.1 of \cite{wall2}, which shows that
$\mu(f)=\displaystyle{\sum_i}\mu(f_i)+2\displaystyle{\sum_{i<j}} m(f_i,f_j)-n+1.$
\end{proof}

\begin{exs}\label{Ex:A1}
{\rm 
(1)	\textit{$A_1$-singularity}. 
Consider a curve $C_f$ with a Morse singularity ($A_1$). We can write $f=gh$, with $C_g, C_h$ regular and transverse. It is easy to see (or use Theorem \ref{theo:f=gh}(5)) that if $f=0, g=0$ do not have an inflection or vertex at $0$, then $I_f=6, V_f=12$.  This can be used to recover Theorem \ref{teo:princ} (1) and (2). For, if we have an irreducible singularity $f=0$, we first perturb the parametrisation to get an immersed curve $C$ with only ordinary double points and simple inflections and vertices, with none at the double points. The number of double points (i.e. $A_1$-singularities) is $\delta(f)=\frac{1}{2}\mu(f)$ (\cite{milnor}). From the calculations above, if we then perturb $C$ to get a smooth curve $C'$, then $C'$ has $I_{\gamma}+6\delta$ inflections and $V_{\gamma}+12\delta$ vertices; that is, $I_f=I_{\gamma}+3\mu(f)$ and $V_f=V_{\gamma}+6\mu(f)$ vertices.

(2) \textit{The Klein cubic}. 
Consider a nodal cubic with a node at $(0:0:1)$ and equation $x^3+y^3-xyz=0$. The Hessian is $3(x^3+y^3)+xyz$, and there are only three inflections, which are collinear. So the double point has absorbed $6$ inflections. We can see how by considering the family
$x^3+y^3+t^3z^3-xyz$, with inflections at the following points:
$$
\begin{array}{lll}
(0:-t:1),& (0:-\omega t:1),& (0:-\omega^2t:1), \\
(t:0:-1),& (t:0:-\omega), &(t:0:-\omega^2), \\
(-1:1:0),& (-\omega:1:0),& (-\omega^2:1:0). 
\end{array}
$$

The first two sets of points give the $6$ inflections emerging from $(0:0:1)$.
}
\end{exs}

Given a singular irreducible germ $f$, let $\lambda(f)$ denote the maximal contact between $C_f$ and circles through the origin, as in Theorem \ref{paramcase}.

\begin{theo}\label{teo:diastari}
{\rm (1)} If $C_f$ is a germ of a singular irreducible curve, then
$$
V_f = I_f+3 \mu(f)+\lambda (f).
$$

{\rm (2)} If $C_f$ is as in $(1)$, then 
$$
\begin{array}{rcccl}
\min\{2m, \beta(f)\}+m+3\mu(f)-3&\le& I_f&\le& \beta(f)+m+3\mu(f) -3,\\[0.2cm]
2\min\{2m,\beta(f)\}+6\mu(f)+m-2&\le& V_f&\le& 2\beta(f)+6\mu(f)+m-2.
\end{array}
$$
{\rm (3)} For a  ${\cal K}$-finite irreducible germ $f:(\C^2,0)\to (\C,0)$,  $I_f$ $($resp. $V_f)$ is constant on the ${\cal K}$-orbit of $f$ if and only if $f$ is singular and $\beta(f)<2m$.

{\rm (4)} More generally, given a ${\cal K}$-finite germ $f:(\C^2,0)\to (\C,0)$,   $I_f$ $($resp. $V_f)$  is constant on the ${\cal K}$-orbit of $f$ if and only if $f$ satisfies the hypotheses of \mbox{\rm Theorem \ref{theo:probs}(7)} and $\beta_1-M_1=\ldots=\beta_k-M_k$. 
\end{theo}

\begin{proof}
(1) This follows from Theorem \ref{teo:princ} and Theorem \ref{paramcase}.

(2) Parametrising $f=0$ as $(t^m,t^n+O(n+1)),$ with $n>m$, we have shown that $I_\gamma=m+n-3=I_f-3\mu(f)$. Clearly, $n$ is either a multiple of $m$ or $\beta(f)$. It follows that $\min\{2m,\beta(f)\}+m-3+3\mu(f) \le I_f$, and the other inequality is now clear. For the second part, from (1), we use the estimates for $\lambda(f)$ given in Theorem \ref{theo:probs}. 

(3) This follows from the fact that $I_\gamma$ is expressed in terms of $m$ and the order of contact, which must have a unique value. For given $N\in {\cal I}(f)$,  by a diffeomorphism we can find a line (resp, circle) whose order of contact with an ${\cal A}$-equivalent germ is of order $N$. The result now follows from Theorem \ref{theo:probs}. 

(4) Again, we need ${\cal I}(f)$ to have at most two values. This clearly implies that for each irreducible component $g$, we have $\beta(g)<2\, \mult(g)$, and the result follows from Theorem \ref{theo:probs}(7). 
\end{proof}

Generically, the values of $I_\gamma$ and $V_\gamma$, or $I_f$ and $V_f$, are minimal, so it is of interest to compute them. First, note that  by Theorem \ref{teo:princ}, if $\gamma$ and $f$ yield the same (irreducible) curve, then  $I_f=I_\gamma+3\mu(f)$ and $V_f=V_\gamma+6\mu(f)$. Thus, the value of $I_\gamma$ (resp. $V_\gamma$) is minimised if and only if  $I_f$ (resp. $V_f$) is minimised. 

For $\gamma:(\C,0)\to (\C^2,0)$ in the form $\gamma(t)=(t^m,t^n+O(n+1)),$ with $ n>m$, we have $I_\gamma=m+n-3$. Suppose then we have an ${\mathcal A}$-type $\gamma$. We know that $n\le \beta(\gamma)$, so the minimal value of $I_\gamma$ is $m+\min\{2m,\beta(\gamma)\}-3$, while the minimal value of $V_\gamma$ is $3m+n-6$ by Theorem \ref{paramcase}(4).

The following result shows that the minimal values attainable for a general germ depend only on its irreducible components.

\begin{theo}\label{theo:minvalsIfVf}
Let $f=f_1\ldots f_k$, where each $f_j$ is irreducible. Let $I_{f_j}^0$ $($resp. $V_{f_j}^0)$ denote the minimal value of $I_{f_j}$ $($resp. $V_{f_j})$, and let $I_{f}^0$ $($resp. $V_{f}^0)$ denote the minimal value of $I_f$ $($resp. $V_f)$. Then:
$$
\begin{array}{l}
I_f^0=\displaystyle{\sum_{j=1}^k I_{f_j}^0+6\sum_{1\le i<j\le k}m(f_i,f_j),} \\
V_f^0=\displaystyle{\sum_{j=1}^k V_{f_j}^0+12\sum_{1\le i<j\le k}m(f_i,f_j).}
\end{array}
$$
\end{theo}
\begin{proof}
We start by considering the set $A^N$ of polynomial maps of degree $\le N$ which are germs of diffeomorphisms $\phi:(\C^2,0)\to (\C^2,0)$; this is the complement of a proper subvariety of a vector space. Given an irreducible curve $\gamma:(\C,0)\to (\C^2,0)$, we can compute the integers $I_{\phi\circ \gamma}$ for each $\phi\in A^N$. Clearly, in this way we obtain the $N$-jet of any germ ${\mathcal C}$-equivalent to $\gamma$. By Theorem \ref{theo:f=gh}, for some $N$ we obtain the minimum value for some $\phi$. Moreover, the set of such $\phi$ is also the complement of a proper algebraic subvariety in $A^N$. If $f$ is the defining equation of $(\gamma(\C),0)$, then $I_{f\circ \phi}$ is minimal. In the general case, we can choose a $\phi$ with $I_{f_i\circ \phi}$ minimal for each $i=1,\ldots, k$. It follows that $I_{f\circ \phi}$ is minimal. The result follows similarly for the minimal value of $V_f$.
\end{proof}


\section{$I_f$ and $V_f$ for simple singularities} \label{sec:exam}

Using the results from the previous sections, it is not difficult to compute the range of values that $I_f$ and $V_f$ can take when $f$ has a simple singularity.
We obtain the following result, omitting, for simplicity, the cases where an irreducible  component of $f$ has a parametrisation of the form $\gamma(t)=(t^m,t^n+O(n+1)),$ with $n=2m$, when computing $V_f$ (see Theorem  \ref{paramcase}(4)).

\begin{theo}\label{theo:Simple_f}
Let $f:(\C^2,0)\to (\C,0)$ be a germ of a holomorphic curve with a simple singularity. Then  the possible values of $I_f$ and $V_f$ for each $\mathcal K$-type of the singularity of $f$ are given in \mbox{\rm Table \ref{tab:IfVfSimpleSing}}, where we adopt the simplifying assumption mentioned above regarding the values of $V_f$  for the $A_{2k}$ and $D_{2k+1}$-singularities. 
\end{theo}

\begin{table}[h!]
\begin{center}
	\caption{$I_f$ and $V_f$ for simple singularities of $f$ $(p,q\in \mathbb N)$.}
	\begin{tabular}{|c|c|c|}
		\hline
		Singularity of $f$ & $I_f$ & $V_f$  \\ 
		\hline
$A_{2k}$, $k \geq 1$ & $I_f=6k+2j-1$, $2 \leq j \leq k$,  & $V_f=12k+2j$, $2 \leq j \leq k$, \\
		& or $I_f=8k$ & or $V_f = 14k+1$\\
		\hline
$A_{2k+1}$, $k \geq 0$ & $I_f=p$, $6k+6 \leq p \leq \infty$ & $V_f=q$, $12k+12 \leq q\leq \infty$ \\ 
		\hline
$D_{2k+1}$, $k \geq 3$ &$I_f=p$,  $6k+9 \leq p \leq \infty$ &$V_f=q$,  $12k+16 \leq q \leq \infty$\\ \hline
$D_{2k}$, $k \geq 2$ &$I_f=p$,  $6k+6 \leq p \leq \infty$ &$V_f=q$,  $ 12k+12\leq q \leq \infty$\\ \hline
$E_6$ & $22$ & $43$\\ \hline
$E_7$ & $I_f=p$, $26 \leq p \leq \infty$ & $V_f=q$, $51 \leq q \leq \infty$ \\ \hline
$E_8$ & $29$ & $56$\\
		\hline
	\end{tabular}\label{tab:IfVfSimpleSing}
\end{center}

\end{table}

We note that, due to the form of $i_f$, we can compute $I_f$ for some germs directly from the formula 
$$I_f=\dim\frac{ {\cal O}_2}{\langle f,i_f\rangle}.$$ 

Suppose that the germ $f:(\C^2,0)\to (\C,0)$ is semi-quasihomogeneous, with quasi-homogeneous part $g$. Thus, for some coprime positive integer weights respectively $w_1$ and $w_2$ assigned to  $x$ and $y$, respectively, the polynomial $g$ has degree $d$, and all the terms of $f-g$ have weight $>d$.  
We assume that neither weight equals $1$, that is $g$ is not homogeneous, and that neither $x$ nor $y$ divides $g$. 
With the same weights, $i_f$ is also semi-weighted homogeneous, of degree $3d-2w_1-2w_2$. Of course, this does not calculate all possible values of $I_f$ within the $\mathcal K$-orbit of $f$.

\begin{prop}
Under the assumptions above, 
$$
I_f=\frac{d(3d-2w_1-2w_2)}{w_1w_2},
$$
and this is the maximum value that can occur for any germ ${\cal K}$-equivalent to $f$.
\end{prop}

\begin{proof}
The map $(x,y)\mapsto (t^{w_1}x,t^{w_2}y)$ preserves $g=0$ and its components. Since neither $x$ nor $y$ divides $g$, it follows that $g$ has no line components. Hence, $I_g$ is finite. Therefore,  $(f,i_f)$ is a semi-quasihomogeneous mapping with finite quasi-homogeneous part $(g,i_g)$, where the degree of the first component is $d$, and the degree of the second is $3d-2w_1-2w_2$. It follows by 
the generalised Bezout formula (see \cite[p. 200]{Arnold}) that 
$$
m(f,i_f)=m(g,i_g)=\frac{d(3d-2w_1-2w_2)}{w_1w_2}.
$$

For the second part, since $g$ is quasi-homogeneous, so is any irreducible factor $g'$. It admits a parametrisation of the form $\gamma(t)=(\alpha t^{w_1},\beta t^{w_2})$ for some $\alpha, \beta,$ with  $\alpha\beta\ne 0$. Assume that $w_1<w_2$; then, as noted above, $I_g=w_1+w_2-3$, which is the maximum possible value for $I_{g'}$. 
The result then follows from Theorem \ref{teo:calc_fat_irred}, since if $g=g_1\ldots g_n$, then the  intersection multiplicities $m(g_i,g_j)=w_1w_2$,  and each $I_{g_i}$ is maximal.
\end{proof}

\begin{ex}
{\rm
(1) Let $f$ be of type $A_{2k}$. By applying a similarity, we may assume  that the lowest order term of $f$ is $x^2$. Assigning weights ${\rm wt}(x)=k+1$ and $ {\rm wt}(y)=2$, suppose that $f$ is  semi-quasihomogeneous with  quasi-homogeneous part $x^2+cy^{2k+1}$, with $c\ne 0$. Clearly, $i_f=8k$. Note that, from Table \ref{tab:IfVfSimpleSing}, this is the largest value possible for an $A_{2k}$-singularity.

(2) More challenging examples are provided by Arnold's extensive lists  (see \cite{Arnold}). For example, for the $W_{11}$-singularity  given by $f(x,y)= x^4+y^5+ax^2y^3$, we have weights $w_1=5, w_2=4$ and degree $d=20$, yielding $I_f=42$. 

For the $J_{k,0}$-singularities,
$$
f(x,y)=x^3+bx^2y^k+y^{3k}+(c_0+\ldots+c_{k-3}y^{k-3})xy^{2k+1}, \,\, 4b^3+27\ne 0, \, \, k\ge 3;
$$ 
we have $w_1=k, w_2=1, d=3k$, and thus $I_f=3(7k-2)$.
}
\end{ex}


\section{Vertices and inflections of curves in $\R^2$} \label{sec:vertR}

Given a real analytic curve $f(x,y)=0$ in the Euclidean plane  $\R^2$, we can consider its complexification 
$f_{c}(x,y)=0$ in $\C^2$,  and define $I_f$ (resp. $V_f$) as $I_{f_{c}}$ (resp. $V_{f_{c}}$). 
Then $I_f$ (resp. $V_f$) provides an upper bound on the number of inflections (resp. vertices) concentrated at the singularity. 
Denote by ${\mathscr R}I_f$ (resp. ${\mathscr R}V_f$ )  the maximum number of real inflections (resp. vertices) concentrated at the singular point that can appear when deforming the curve. Thus, 
${\mathscr R}I_f\le I_f$ and ${\mathscr R}V_f\le V_f$. If $f$ is merely smooth, then provided $(f,i_f)$ (resp. $(f,v_f)$) is a finite mapping, we still obtain a well-defined upper bound, and these are finite maps for all $f$ off a set of infinite codimension.

It is shown in \cite{DiattaGiblin,MarcoSamuel} that for an $A_1^+$-singularity (i.e., $f\sim_\mathcal R \pm (x^2+y^2)$), ${\mathscr R}I_f=0$ and ${\mathscr R}V_f=4$, and  
for an $A_1^-$-singularity (i.e., $f\sim_\mathcal R \pm (x^2-y^2)$), ${\mathscr R}I_f=2$ and ${\mathscr R}V_f=6$. 
Since we have $I_{f_c}=6$ and $V_{f_c}=12$ for an $A_1$-singularity (following the arguments 
in Example \ref{Ex:A1}), this implies that for each double point, we need to remove at least 4 inflections and at least 6 vertices. 
This suggests that for an irreducible germ $f$,   
${\mathscr R}I_f\le I_f-4\delta= I_f-2\mu$ and  ${\mathscr R}V_f\le V_f-6\delta=V_f-3\mu$.

There is a well-defined notion of the degree of a real map germ $F:(\R^2,0)\to (\R^2,0)$, which is given by the topological degree of the mapping $F/||F||: S^{\epsilon}\to S^1$, where $S^1$ is the oriented unit circle, and  $S^{\epsilon}$  is the positively oriented circle of radius $\epsilon$, both oriented consistently with $\mathbb R^2$. If $F$ is differentiable, the degree can be computed as the sum of the signs of the Jacobian determinant of $F$ at all preimages of a regular value sufficiently close to the origin. For an explicit formula, see \cite{Eisenbud}.

If we consider the map-germ $F=(f,i_f)$, then an ordinary inflection has degree (or index) $\pm 1$. 
More precisely, if we choose a local frame given by the tangent and normal vectors to the curve, the is index $+1$ (resp. $-1$) if the curve near the inflection lies in the first and third (respectively, second and fourth) quadrants; see Figure~\ref{fig:infIndVert}.

For vertices, considering the map-germ $F=(f,v_f)$, an ordinary vertex also has index $\pm 1$. 
In \cite{SalariTari}, the notions of inward and outward vertices are defined based on the relative position of the evolute; see Figure~\ref{fig:infIndVert}. A vertex is inward if $\kappa\kappa''>0$, and outward otherwise. 
If we choose a frame given by the tangent and normal vectors to the curve, then the index is $+1$ at an inward vertex and $-1$ at an outward vertex when $\kappa(0)>0$; the signs are reversed when $\kappa(0)<0$.

For a singular germ, the degree of $(f,i_f)$ (resp. $(f,v_f)$) gives the sum of the indices of the ordinary inflections (resp. vertices) that appear in a deformation of $f$.
Of course the degree does not give information about the number of inflections or vertices, only their net contribution via their index.

Using the terminology of Tougeron (see \cite{wall}), we have the following result, whose proof is omitted.

\begin{prop}\label{theo:GenericIV}
The set of smooth germs $\gamma:(\K,0)\to (\K^2,0)$, respectively   $f:(\K^2,0)\to(\K,0)$, for which $I_\gamma$ or $V_\gamma$, respectively  $I_f$ or $ V_f$, is not finite, is of infinite codimension.
\end{prop}

\begin{figure}[tp]
\begin{center}
\includegraphics[scale=2]{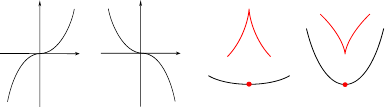}
\caption{The degree of an ordinary inflection, $+1$ (first figure) and $-1$ (second figure); 
	inward vertex (third figure)  and outward vertex (last figure). The curve in red is the evolute of the curve in black.}
	\label{fig:infIndVert}
\end{center}
\end{figure}

\begin{acknow}
The work in this paper was  supported by the FAPESP Thematic project grant 2019/07316-0.
The second author was supported by the FAPESP postdoctoral grant 2022/06325-8.
\end{acknow}


\noindent 
JWB: Department of Mathematical Sciences, University of Liverpool, Liverpool, L69 3BXl\\
E-mail: billbrucesingular@gmail.com\\

\noindent
MACF: Departamento de Matem\'atica, Universidade Federal de Vi\c{c}osa, Av. P. H. Rolffs, s.n., Campus Universit\'ario, CEP: 36570-000, Vi\c{c}osa - MG, Brazil
\\
E-mail: marco.a.fernandes@ufv.br\\

\noindent 
FT: Instituto de Ci\^encias Matem\'aticas e de Computa\c{c}\~ao - USP, Avenida Trabalhador s\~ao-carlense, 400, Centro, CEP: 13566-590, S\~ao Carlos - SP, Brazil.
\\
E-mail: faridtari@icmc.usp.br
\end{document}